\documentclass[12pt, reqno]{amsart}
\usepackage{amsthm,amssymb}
\usepackage{graphicx}
\usepackage{enumerate}
\usepackage{amssymb,amsmath,graphicx,amsfonts,euscript}
\usepackage{color}

\usepackage[margin=3cm, a4paper]{geometry}

\def\H{{\cal H}}

\def\H1{H^1(\R)}

 \newcommand{\R}{\mathbb{R}}

 \newcommand{\Del}[1]{}

\numberwithin{equation}{section}

\newtheorem{thm}{Theorem}[section]
\newtheorem{cor}[thm]{Corollary}
\newtheorem{lem}[thm]{Lemma}

\newtheorem{definition}[thm]{Definition}

\theoremstyle{remark}

\newtheorem*{exam*}{Examples}

\begin{document}

\setcounter{page}{1}

\title[Instability of standing waves]{Instability of the standing waves for the nonlinear Klein-Gordon equations in one dimension}

%

\author[Yifei Wu]{Yifei Wu}
\address{Center for Applied Mathematics, Tianjin University,
Tianjin 300072, P.R.China}
\email{yerfmath@gmail.com}
\subjclass[2010]{Primary  35L70; Secondary 35B35}


\keywords{Nonlinear Klein-Gordon equation,
instability, standing waves, critical frequency, one dimension}

\maketitle

\begin{abstract}\noindent
In this paper, we consider the following nonlinear Klein-Gordon equation
\begin{align*}
\partial_{tt}u-\Delta u+u=|u|^{p-1}u,\qquad t\in \R,\ x\in \R^d,
\end{align*}
with $1<p<  1+\frac{4}{d}$. The equation has the standing wave solutions $u_\omega=e^{i\omega t}\phi_{\omega}$ with the frequency $\omega\in(-1,1)$, where $\phi_{\omega}$ obeys
\begin{align*}
-\Delta \phi+(1-\omega^2)\phi-\phi^p=0.
\end{align*}
It was proved by Shatah (1983), and Shatah, Strauss (1985) that there exists a critical frequency $\omega_c\in (0,1)$ such that  the standing waves solution $u_\omega$ is orbitally stable when $\omega_c<|\omega|<1$, and orbitally unstable when $|\omega|<\omega_c$. Further, the critical case $|\omega|=\omega_c$ in the high dimension $d\ge 2$ was considered by Ohta, Todorova (2007), who proved that it is strongly unstable, by using the virial identities and the radial Sobolev inequality. The one dimension problem was left after then. In this paper, we consider the one-dimension problem and prove that it is orbitally unstable when $|\omega|=\omega_c$.
\end{abstract}

\section{Introduction}
In this paper, we consider the stability theory of the following nonlinear Klein-Gordon equation
\begin{align}\label{eq:NKG}
\begin{split}
&\partial_{tt}u-\Delta u+u=|u|^{p-1}u,\qquad t\in \R,\ x\in \R^d,
\end{split}
\end{align}
with the initial data
\begin{align}\label{eq:Initialdata}
u(0,x) =u_0(x),\quad u_t(0,x) =u_1(x).
\end{align}
Here $d\ge 1$ and $1<p<  1+\frac{4}{d-2}$ ($1<p<\infty$ when $d=1,2$). The $H^1\times L^2$-solution $(u,u_t)$ of \eqref{eq:NKG}--\eqref{eq:Initialdata} obeys the following charge, momentum and energy conservation laws,
\begin{align}
Q(u,u_t)&=\mbox{Im}\int u \bar u_t\,dx=Q(u_0,u_1);\label{charge}\\
P(u,u_t)&=\mbox{Re}\int\nabla u \bar u_t\,dx=P(u_0,u_1);\label{Mom}\\
E(u,u_t)&=\|u_t\|_{L^2}^2+\|\nabla u\|_{L^2}^2+\|u\|_{L^2}^2-\frac2{p+1}\|u\|_{L^{p+1}}^{p+1}=E(u_0,u_1).\label{Energy}
\end{align}
The well-posedness for the Cauchy problem \eqref{eq:NKG}--\eqref{eq:Initialdata} was well understand in the energy space $  H^1(\R^d)\times L^2(\R^d)$. More precisely, for any $(u_0,u_1)\in   H^1(\R^d)\times L^2(\R^d)$, there exists a unique solution $(u,u_t)\in C([0,T);  H^1(\R^d)\times L^2(\R^d))$ of \eqref{eq:NKG}--\eqref{eq:Initialdata}, with the maximal lifetime $T=T(\|(u_0,u_1)\|_{H^1\times L^2})$. If $T=\infty$, we call that the Cauchy problem \eqref{eq:NKG}--\eqref{eq:Initialdata} is global well-posedness. If $T<\infty$, we call that the solution blows up in finite time. See for examples Ginibre and Velo \cite{GiniVelo-MZ-85, GiniVelo-AIHPAN-89} for the local and global well-posedness, and Payne and Sattinger \cite{PaSa-IJM-75} for the blowing-up. Further results on the scattering, see \cite{IbMaNa-APDE-11, Inui-AHP-17, KiStVi-TAMS-10, MiZh-DCDS-2016} and the references therein.

The equation \eqref{eq:NKG} has the standing waves solution $e^{i\omega t}\phi_{\omega}$, where $\phi_{\omega}$ is the ground state solution of the following elliptic equation
\begin{align}\label{Elliptic}
-\Delta \phi+(1-\omega^2)\phi-\phi^p=0.
\end{align}
The equation \eqref{Elliptic} exists solutions when the parameter $|\omega|<1$, see \cite{Strauss-CMP-77} for example. In particular, in one dimension case, the solution to \eqref{Elliptic} is unique up to the symmetries of the rotation and the spatial transformation. Moreover, the ground state solution $\phi_{\omega}$ is exponential decaying at infinity when $|\omega|<1$. See also \cite{BeGhLecoz-CPDE-14,CoLecoz-JMPA-11,CoMa-CPDE-17,CoMaMe-RMI-11,CoMu-FM-14,LecozLiTsai-PRSEA-15} for some instances on the existence of the multi-solitary waves of the nonlinear Klein-Gordon and the nonlinear Schr\"odinger equations.

The stability theory of the the standing waves solution $e^{i\omega t}\phi_{\omega}$ has been widely studied. In particular, Berestycki and Cazenave \cite{BeCa-CRASP-81} proved that it is strong instability when $\omega=0$ and $1<p< 1+\frac{4}{d-2}$, which is in the sense that an arbitrarily small perturbation of the initial data can find the perturbed solution blowing up in finite time. See also Shatah \cite{Shatah-TAMS-85} for the related works. One may find the big difference between the nonlinear  Klein-Gordon equation and the nonlinear Schr\"odinger equation, because of the lack of the mass conservation law.
Further, when $\omega\ne 0$, Shatah \cite{Shatah-CMP-83} proved that it is orbital stability when $1<p<1+\frac{4}d$ and $\omega_c<|\omega|<1$, where the frequency $\omega_c$ is equal to
$$
\omega_c=\sqrt{\frac{p-1}{4-(d-1)(p-1)}}.
$$
The number $\omega_c$ is critical. Indeed, Shatah and Strauss \cite{ShatahStrauss-CMP-85} showed further that when $1<p<1+\frac{4}d, |\omega|<\omega_c$ or $1+\frac{4}d<p<1+\frac{4}{d-2}, |\omega|<1$, the standing waves solution $e^{i\omega t}\phi_{\omega}$ is orbital instability. See also Stuart \cite{Stuart-JMPA-01} for the stability of the solitary waves. The critical cases, $|\omega|=\omega_c$ when $1<p<1+\frac{4}d$, are degenerate based on the theories of Grillakis, Shatah and Strauss \cite{GrShStr-87, GrShStr-90}. The degenerate cases were further investigated by several authors, such as Comech, Pelinovsky \cite{CoPe-CPAM-03}, Maeda \cite{Maeda-NA-10, Maeda-JFA-12}, and Ohta \cite{Ohta-JFA-11}. In particular, as an application of the theorems established in \cite{CoPe-CPAM-03, Maeda-JFA-12},  the standing waves solution $e^{i\omega t}\phi_{\omega}$ is orbitally unstable in the critical cases $|\omega|=\omega_c$ when $2\le p<1+\frac{4}d$. The region $1<p<2$ was not covered because of the lack of the regularity for the relevant functionals.
Further, Ohta and Todorova \cite{OhtaTod-DCDS-05, OhtaTod-SIAM-07} (see also
 \cite{JeLeC-TAMS-09} for a companion result) proved the strong instability when $d\ge 2, 1<p<1+\frac{4}d, |\omega|\le \omega_c$ or $d\ge 2, 1+\frac{4}d\le p<1+\frac{4}{d-2}, |\omega|<1$,  which cover the entire instability region in the case of $d\ge 2$. The argument the authors used was the variation argument combining with the virial identities. Hence, the stability and instability regions were complete division except the one dimension cases, and the only left problem is the stability theory of the soliton in the case of $1<p<2, |\omega|=\omega_c$ when $d=1$. Unfortunately, the argument in \cite{OhtaTod-SIAM-07} is not available in one dimension problem, because the argument relies on the radial choice of the instable data, which gives the small control of the remainder terms from the localized virial identities by the radial Sobolev inequality. In one dimension, Liu, Ohta and Todorova \cite{LiuOhtaTod-AIPANL-07} considered the strong instability in some regions which still has gap from the critical frequency. In present paper, we study the instability of the standing waves solution in the critical case in one dimension.

Before stating our theorem, we recall some definitions. Let $v=u_t$, $\vec u=(u,v)^T$, $\vec u_0=(u_0,u_1)^T$,  and  $\overrightarrow{\Phi_\omega}=(\phi_\omega,i\omega\phi_\omega)^T$.
For $\varepsilon>0$, we denote the set $ U_\varepsilon\big(\overrightarrow{\Phi_\omega}\big)$ as
$$ U_\varepsilon\big(\overrightarrow{\Phi_\omega}\big)
=\{\vec u\in H^1(\mathbb{R})\times L^2(\R): \inf_{(\theta,y)\in\mathbb{R}^2}\|\vec u-e^{i\theta}\overrightarrow{\Phi_\omega}(\cdot-y)\|_{H^1\times L^2}<\varepsilon\}.$$

\begin{definition}\label{Def:stability}
We say that the solitary wave solution $u_\omega$ of \eqref{eq:NKG} is stable
if for any $\varepsilon >0$ there exists $\delta >0$ such that if $\|\vec u_0-\overrightarrow{\Phi_\omega}\|_{H^1\times L^2}< \delta$,
then the solution $\vec u(t)$ of \eqref{eq:NKG} with $\vec u(0)=\vec u_0$ exists for all $t\in\R$,
and $\vec u(t)\in U_\varepsilon\big(\overrightarrow{\Phi_\omega}\big)$ for all $t\in\R$.
Otherwise, $u_\omega$ is said to be unstable.
\end{definition}

Then the main result in the present paper is
\begin{thm}\label{thm:main1}
Let $d=1$, $1<p<5$, $\omega\in (-1,1)$ and $\phi_\omega$ be a solution of \eqref{Elliptic}. If $|\omega|=\sqrt{\frac{p-1}4}$,
then the standing waves solution $e^{i\omega t}\phi_\omega$ is orbitally unstable.
\end{thm}

The method used to prove the theorem is the modulation argument combining with the virial identity, which is completely different from \cite{CoPe-CPAM-03, Maeda-JFA-12, OhtaTod-SIAM-07} and inspired in the work of \cite{MaMe-GFA-2001}. The modulation argument used here was introduced by Weinstein \cite{W2}, and strengthened by  the mathematicians such as Martel, Merle, Rapha\"{e}l \cite{MaMe-GFA-2001, Me-JAMS-2001, MePa-IV-2004}.  In particular, in the Klein-Gordon setting, we use the modulation method applied by Bellazzini1, Ghimenti, and
Le Coz \cite{BeGhLecoz-CPDE-14}, who considered the total linearized action. The modulation argument is much problem dependent. 
Rough speaking, we argue for contradiction and suppose that the solution is close to the standing wave solution in the whole time, then the modulation argument gives the smallness of the perturbation up to the rotation, spatial transformation and scaling. Ultimately, we use the local virial identity to preclude that scaling parameter
always keeps near the initial size in the whole time.  In particular, the smallness of the perturbation gives the tiny estimates of the remainder terms from the local virial identity. Then the control of the scaling parameter become one of the key ingredients in the proof of the theorem.
In the present paper,   we utilize the flatness of functional $E\big(\overrightarrow{\Phi_\omega}\big)+\lambda\omega Q\big(\overrightarrow{\Phi_\omega}\big)$ in $\lambda$
to establish the high order control of the the scaling parameter $\lambda$;
and under the contradictory hypotheses, we utilize the term $\|v-i\omega u\|_{L^2}^2$ from the virial identity, the charge conservation law and appropriately choose the orthogonal condition in the coercivity lemma to give the upper control the scaling parameter.

It is worth noting that our argument used here does not rely on the regularity of the nonlinear term. 
Further, we believe the strong instability is true in our case and our argument here could be used to prove the strong result. It leaves us an interesting problem to pursue in the future.

Now the following is the organization of the paper. In Section 2, we give some preliminaries. It includes some basic definitions and properties, the coercivity property of the Hessian, and the modulation statement.  In Section 3, we give the virial identities, control the remainder function and the scaling parameter, and lastly  prove the main theorem.

\vskip 2cm

\section{Preliminary}

\vskip .2cm

\subsection{Notations}

For $f,g\in L^2(\mathbb{R})=L^2(\mathbb{R,C})$, we define
$$\langle f,g\rangle=\mbox{Re}\int_{\mathbb{R}}f(x)\overline{g(x)}\,dx$$
and regard $L^2(\mathbb{R})$ as a real Hilbert space. Similarly, for $\vec f,\vec g\in \big(L^2(\mathbb{R})\big)^2=\big(L^2(\mathbb{R,C})\big)^2$, we define
$$\langle \vec f,\vec g\rangle=\mbox{Re}\int_{\mathbb{R}}\vec f(x)^T\cdot \overline{\vec g(x)}\,dx.$$

For a function $f(x)$, its $L^{q}$-norm $\|f\|_{L^q}=\Big(\displaystyle\int_{\mathbb{R}} |f(x)|^{q}dx\Big)^{\frac{1}{q}}$
and its $H^1$-norm $\|f\|_{H^1}=(\|f\|^2_{L^2}+\|\partial_x f\|^2_{L^2})^{\frac{1}{2}}$.

Further, we write $X \lesssim Y$ or $Y \gtrsim X$ to indicate $X \leq CY$ for
some constant $C>0$.  We use the notation $X \sim Y$ whenever $X
\lesssim Y \lesssim X$. Also, we use $O(Y)$ to denote any quantity $X$ such
that $|X| \lesssim Y$; and use $o(Y)$ to denote any quantity $X$ such
that $X\to 0$, if $ Y\to 0$.

\subsection{Some basic definitions and properties}
In the following, we only consider one dimension problem and the case of $1<p<5$, in which $\omega_c=\sqrt{\frac{p-1}4}$.
Let $\vec u=(u,v)^T, \overrightarrow{\Phi_\omega}=(\phi_\omega,i\omega\phi_\omega)^T$. Recall that the conserved qualities,
\begin{align*}
Q(\vec u)&=\mbox{Im}\int u \bar u_t\,dx,\\
E(\vec u)&=\|u_t\|_{L^2}^2+\| u_{x}\|_{L^2}^2+\|u\|_{L^2}^2-\frac2{p+1}\|u\|_{L^{p+1}}^{p+1}.
\end{align*}
First, we give some basic properties on the charge and energy.
\begin{lem}\label{lem:QE-Phi} The following equalities hold,
\begin{itemize}
  \item[(1)] $
\frac{d}{d\omega}Q\left(\overrightarrow{\Phi_\omega}\right)\Big|_{\omega=\pm\omega_c}=0;
$
  \item[(2)] If $|\omega|=\omega_c$, then $(p+3)E\left(\overrightarrow{\Phi_\omega}\right)+8\omega Q\left(\overrightarrow{\Phi_\omega}\right)=0$.
\end{itemize}
\end{lem}
\begin{proof}
Note that
$$
Q\left(\overrightarrow{\Phi_\omega}\right)=-\omega \|\phi_\omega\|_{L^2}^2.
$$
Moreover, by rescaling, we find,
$$
\phi_\omega(x)=(1-\omega^2)^\frac1{p-1}\phi_0\big(\sqrt{1-\omega^2}x\big).
$$
This implies that
$$
Q\left(\overrightarrow{\Phi_\omega}\right)=-\omega (1-\omega^2)^{\frac2{p-1}-\frac12} \|\phi_0\|_{L^2}^2.
$$
Hence by a direct computation, we have
$$
\frac{d}{d\omega}Q\left(\overrightarrow{\Phi_\omega}\right)
=-(1-\omega^2)^{\frac2{p-1}-\frac32}\Big[1-\frac4{p-1}\omega^2\Big]\|\phi_0\|_{L^2}^2.
$$
This gives (1). For (2), we have
$$
E\left(\overrightarrow{\Phi_\omega}\right)=\|\partial_x\phi_\omega\|_{L^2}^2+(1+\omega^2)\|\phi_\omega\|_{L^2}^2-\frac2{p+1}\|\phi_\omega\|_{L^{p+1}}^{p+1}.
$$
From the equation \eqref{Elliptic}, we obtain that
\begin{align*}
&\|\partial_x\phi_\omega\|_{L^2}^2+(1-\omega^2)\|\phi_\omega\|_{L^2}^2-\|\phi_\omega\|_{L^{p+1}}^{p+1}=0;\\
&\|\partial_x\phi_\omega\|_{L^2}^2-(1-\omega^2)\|\phi_\omega\|_{L^2}^2+\frac2{p+1}\|\phi_\omega\|_{L^{p+1}}^{p+1}=0.
\end{align*}
These give that
$$
E\left(\overrightarrow{\Phi_\omega}\right)=\frac1{p+3}\big(p-1+4\omega^2\big)\|\phi_\omega\|_{L^2}^2.
$$
Combining the value of $Q\left(\overrightarrow{\Phi_\omega}\right)$ above, we obtain (2).
\end{proof}

Now we define the functional $S_\omega$ as
$$
S_\omega(\vec u)=E(\vec u)+\omega Q(\vec u).
$$
Then we have
\begin{align*}
S_\omega'(\vec u)=2\left(
                    \begin{array}{c}
                      -u_{xx}+u-|u|^{p-1}u \\
                      v \\
                    \end{array}
                  \right)
                  +2i\omega
                  \left(
                    \begin{array}{c}
                      v \\
                     -u \\
                    \end{array}
                  \right).
\end{align*}
Note that $S_\omega'(\overrightarrow{\Phi_\omega})=0$. Moreover, for the vector $\vec f=(f,g)^T$, a direct computation shows that
\begin{align}\label{23.06}
S_\omega''\big(\overrightarrow{\Phi_\omega}\big)\vec f=2\left(
                    \begin{array}{c}
                      -f_{xx}+f-p\phi_\omega^{p-1}\mbox{Re} f-i\phi_\omega^{p-1}\mbox{Im} f \\
                      g \\
                    \end{array}
                  \right)
                  +2i\omega
                  \left(
                    \begin{array}{c}
                      g \\
                     -f \\
                    \end{array}
                  \right).
\end{align}
From the invariance of $S_\omega'\big(\overrightarrow{\Phi_\omega}\big)$ in the rotation and spatial transformations, we have
\begin{align}\label{23.07}
S_\omega''\big(\overrightarrow{\Phi_\omega}\big) i\overrightarrow{\Phi_\omega}=0,\qquad
S_\omega''\big(\overrightarrow{\Phi_\omega}\big) \partial_x\overrightarrow{\Phi_\omega}=0.
\end{align}
Indeed, from
$$
S_\omega'\big(e^{i\theta}\overrightarrow{\Phi_\omega}(\cdot-y)\big)=0,\quad\mbox{ for any } \theta\in\R, y\in\R,
$$
we find that
$$
S_\omega''\big(\overrightarrow{\Phi_\omega}\big)i\overrightarrow{\Phi_\omega}=\partial_\theta S_\omega'\big(e^{i\theta}\overrightarrow{\Phi_\omega}\big)\Big|_{\theta=0}=0,
$$
and
$$
S_\omega''\big(\overrightarrow{\Phi_\omega}\big)\partial_x\overrightarrow{\Phi_\omega}=-\partial_y S_\omega'\big(\overrightarrow{\Phi_\omega}(\cdot-y)\big)\Big|_{y=0}=0.
$$
This gives \eqref{23.07}.

Moreover, taking the derivative of $S_\omega'\big(\overrightarrow{\Phi_\omega}\big)=0$ gives that
\begin{align}
S_\omega''\big(\overrightarrow{\Phi_\omega}\big)\partial_\omega\overrightarrow{\Phi_\omega}=-Q'\big(\overrightarrow{\Phi_\omega}\big).\label{14.10}
\end{align}
Then a consequence of Lemma \ref{lem:QE-Phi} (1) is
\begin{cor}\label{cor:S-lambda} Let $\lambda\in \R^+$, $\omega=\pm\omega_c$, then
\begin{align*}
S_{\lambda\omega}\big(\overrightarrow{\Phi_{\lambda\omega}}\big)-S_{\lambda\omega}\big(\overrightarrow{\Phi_{\omega}}\big)=o\big((\lambda-1)^2\big).
\end{align*}
\end{cor}
\begin{proof}
From the definition and the Taylor's type expansion,
\begin{align*}
S_{\lambda\omega}\big(\overrightarrow{\Phi_{\lambda\omega}}\big)&-S_{\lambda\omega}\big(\overrightarrow{\Phi_{\omega}}\big)\\
=&
S_{\omega}\big(\overrightarrow{\Phi_{\lambda\omega}}\big)-S_{\omega}\big(\overrightarrow{\Phi_{\omega}}\big)
+(\lambda-1)\omega \Big(Q\big(\overrightarrow{\Phi_{\lambda\omega}}\big)-Q\big(\overrightarrow{\Phi_{\omega}}\big)\Big)\\
&=
\frac12\left\langle S_{\omega}''\big(\overrightarrow{\Phi_{\omega}}\big)\Big(\overrightarrow{\Phi_{\lambda\omega}}-\overrightarrow{\Phi_{\omega}}\Big),
\Big(\overrightarrow{\Phi_{\lambda\omega}}-\overrightarrow{\Phi_{\omega}}\Big)\right\rangle\\
&\qquad
+(\lambda-1)\omega \Big(Q\big(\overrightarrow{\Phi_{\lambda\omega}}\big)-Q\big(\overrightarrow{\Phi_{\omega}}\big)\Big)+o\big((\lambda-1)^2\big).
\end{align*}
Note that
$$
\overrightarrow{\Phi_{\lambda\omega}}-\overrightarrow{\Phi_{\omega}}=(\lambda-1)\omega \partial_\omega\overrightarrow{\Phi_\omega}+o(\lambda-1),
$$
we find that
\begin{align*}
\Big\langle S_{\omega}''\big(\overrightarrow{\Phi_{\omega}}\big)&\Big(\overrightarrow{\Phi_{\lambda\omega}}-\overrightarrow{\Phi_{\omega}}\Big),
\Big(\overrightarrow{\Phi_{\lambda\omega}}-\overrightarrow{\Phi_{\omega}}\Big)\Big\rangle\\
= &
(\lambda-1)^2\omega^2\left\langle S_{\omega}''\big(\overrightarrow{\Phi_{\omega}}\big)\partial_\omega\overrightarrow{\Phi_\omega},
\partial_\omega\overrightarrow{\Phi_\omega}\right\rangle+o\big((\lambda-1)^2\big)\\
= &
-(\lambda-1)^2\omega^2\left\langle Q'\big(\overrightarrow{\Phi_\omega}\big),
\partial_\omega\overrightarrow{\Phi_\omega}\right\rangle+o\big((\lambda-1)^2\big)\\
= &
-(\lambda-1)^2\omega^2\frac{d}{d\lambda} Q\big(\overrightarrow{\Phi_{\lambda\omega}}\big)\Big|_{\lambda=1}+o\big((\lambda-1)^2\big),
\end{align*}
here we have used \eqref{14.10} in the second step.
Using Lemma \ref{lem:QE-Phi} (1), we have
$$
\frac{d}{d\lambda} Q\big(\overrightarrow{\Phi_{\lambda\omega}}\big)\Big|_{\lambda=1}=0.
$$
Hence,
$$
\Big\langle S_{\omega}''\big(\overrightarrow{\Phi_{\omega}}\big)\Big(\overrightarrow{\Phi_{\lambda\omega}}-\overrightarrow{\Phi_{\omega}}\Big),
\Big(\overrightarrow{\Phi_{\lambda\omega}}-\overrightarrow{\Phi_{\omega}}\Big)\Big\rangle=o\big((\lambda-1)^2\big),
$$
and
$$
Q\big(\overrightarrow{\Phi_{\lambda\omega}}\big)-Q\big(\overrightarrow{\Phi_{\omega}}\big)=o\big(\lambda-1\big).
$$
Thus we obtain the desirable estimate.
\end{proof}

\subsection{Coercivity}
First, we need the following lemma.
\begin{lem}\label{lem:negative}
Let $\vec \psi_\omega=(\partial_\omega\phi_\omega, i\omega\partial_\omega\phi_\omega)^T$, $\overrightarrow{\Psi_\omega}=(4\omega\phi_\omega,0)^T$, then
\begin{align}
S_\omega''\big(\overrightarrow{\Phi_\omega}\big)\vec \psi_\omega=\overrightarrow{\Psi_\omega}.\label{23.34}
\end{align}
Moreover, if $|\omega|=\omega_c$, then
\begin{align}
\left\langle S_\omega''\big(\overrightarrow{\Phi_\omega}\big)\vec \psi_\omega, \vec \psi_\omega\right\rangle<0.\label{negative}
\end{align}
\end{lem}
\begin{proof}
Note that from the equation \eqref{Elliptic}, we have
$$
\big(-\partial_{xx}+(1-\omega^2)-p\phi_\omega^{p-1}\mbox{Re} -i\phi_\omega^{p-1}\mbox{Im} \big)\partial_\omega\phi_\omega=2\omega \phi_\omega.
$$
Then \eqref{23.34} follows from a straightforward computation.

For \eqref{negative}, we have
\begin{align*}
\left\langle S_\omega''\big(\overrightarrow{\Phi_\omega}\big)\vec \psi_\omega, \vec \psi_\omega\right\rangle
&=\left\langle \overrightarrow{\Psi_\omega}, \vec \psi_\omega\right\rangle\\
&=4\omega\int \phi_\omega\>\partial_\omega\phi_\omega\,dx=2\omega\frac{d}{d \omega}\|\phi_\omega\|_{L^2}^2\\
&\qquad =-2\frac{d}{d \omega} Q\left(\overrightarrow{\Phi_\omega}\right)-2\|\phi_\omega\|_{L^2}^2.
\end{align*}
Using Lemma \ref{lem:QE-Phi} (1), when $|\omega|=\omega_c$,
\begin{align*}
\left\langle S_\omega''\big(\overrightarrow{\Phi_\omega}\big)\vec \psi_\omega, \vec \psi_\omega\right\rangle
=-2\|\phi_\omega\|_{L^2}^2<0.
\end{align*}
This proves the lemma.
\end{proof}

Now we have the following coercivity property.
\begin{lem}\label{lem:Coercivity}
Let $\omega=\pm\omega_c$. Suppose that $\vec \xi \in H^1(\R)\times L^2(\R)$ satisfies
$$
\left\langle\vec \xi ,i\overrightarrow{\Phi_\omega}\right\rangle
=\left\langle\vec \xi, \partial_x\overrightarrow{\Phi_\omega}\right\rangle
=\left\langle\vec \xi ,\overrightarrow{\Psi_\omega}\right\rangle=0.
$$
Then
$$
\left\langle S_\omega''\big(\overrightarrow{\Phi_\omega}\big)\vec \xi, \vec \xi\right\rangle
\gtrsim \big\|\vec \xi\big\|_{H^1\times L^2}^2.
$$
\end{lem}
\begin{proof}
First, we show that
\begin{align}
\mbox{Ker}\Big( S_\omega''\big(\overrightarrow{\Phi_\omega}\big)\Big)=\mbox{Span}\big\{i\overrightarrow{\Phi_\omega}, \partial_x\overrightarrow{\Phi_\omega}\big\}.\label{19.47}
\end{align}
Indeed, from \eqref{23.07}, we have
$$
\big\{i\overrightarrow{\Phi_\omega}, \partial_x\overrightarrow{\Phi_\omega}\big\}\subset \mbox{Ker}\Big( S_\omega''\big(\overrightarrow{\Phi_\omega}\big)\Big).
$$
Hence, to prove \eqref{19.47}, we now turn to show that if
\begin{align}
S_\omega''\big(\overrightarrow{\Phi_\omega}\big)\vec f=0, \label{23.00}
\end{align}
then
\begin{align}
\vec f=c_1i\overrightarrow{\Phi_\omega}+c_2\partial_x\overrightarrow{\Phi_\omega}.\label{23.42}
\end{align}
Let $\vec f=(f,g)$, then from \eqref{23.06}, the equality \eqref{23.00} is equivalent to
\begin{align*}
\left\{
                    \begin{array}{l}
                      -f_{xx}+f-p\phi_\omega^{p-1}\mbox{Re} f-i\phi_\omega^{p-1}\mbox{Im} f+i\omega g=0, \\
                      g-i\omega f=0 \\
                    \end{array}
                  \right.
\end{align*}
This implies that $f$ obeys the equation
$$
-f_{xx}+(1-\omega^2)f-p\phi_\omega^{p-1}\mbox{Re} f-i\phi_\omega^{p-1}\mbox{Im} f=0.
$$
Then from Proposition 2.8 in Weinstein \cite{W2}, we obtain that there exist $c_1\in\R, c_2\in\R$,
$$
f=c_1 \partial_x \phi_\omega+c_2i\phi_\omega.
$$
This yields that
$$
g=i\omega f= c_1i\omega \partial_x \phi_\omega+c_2i\omega \cdot i \phi_\omega.
$$
Hence we have \eqref{23.42} and thus we prove \eqref{19.47}.

Second, we claim that
\begin{align}
S_\omega''\big(\overrightarrow{\Phi_\omega}\big) \mbox{ has exactly one
negative eigenvalue}.\label{0.02}
\end{align}
To prove \eqref{0.02}, we need some well-known facts. It is known that the operator
\begin{align}
-\partial_{xx}+(1-\omega^2)-\phi_\omega^{p-1}\quad \mbox{is non-negative},
\label{0.03}
\end{align}
and the operator
$$
-\partial_{xx}+(1-\omega^2)-p\phi_\omega^{p-1}
$$
has exactly one negative eigenvalue (see Page 489 in Weinstein \cite{W2}).  That is, there uniquely exists a pair $(\lambda_{-1},f_{-1})\in \R^{-}\times H^1(\R)$ such that
$$
-\partial_{xx}f_{-1}+(1-\omega^2)f_{-1}-p\phi_\omega^{p-1}f_{-1}=\lambda_{-1}f_{-1}.
$$
Moreover, the formula \eqref{negative} implies that $S_\omega''\big(\overrightarrow{\Phi_\omega}\big)$ at least exists one negative eigenvalue.
That is, there is at least one negative eigenvalue and its associated eigenvector, say $(\tilde\lambda_{-1},\vec \xi_{-1})\in \R^{-}\times H^1(\R)$, such that
\begin{align}
S_{\omega}''(\overrightarrow{\Phi_\omega})\vec \xi_{-1}=\tilde\lambda_{-1} \vec \xi_{-1}.\label{0.04}
\end{align}
Using \eqref{23.06}, this yields that
\begin{align*}
\left\{
                    \begin{array}{l}
           -\partial_{xx}\xi_{-1}+\xi_{-1}-p\phi_\omega^{p-1}\mbox{Re} \xi_{-1}-i\phi_\omega^{p-1}\mbox{Im} \xi_{-1}+i\omega \eta_{-1}=\tilde\lambda_{-1}\xi_{-1},\\
           \eta_{-1}-i\omega \xi_{-1}=\tilde\lambda_{-1}\eta_{-1},
                    \end{array}
                  \right.
\end{align*}
where $\vec \xi_{-1}=(\eta_{-1},\eta_{-1})$.
This further implies that
\begin{align*}
\left\{
                    \begin{array}{l}
           -\partial_{xx}\xi_{-1}+(1-{\omega}^2)\xi_{-1}-p\phi_\omega^{p-1}\mbox{Re} \xi_{-1}-i\phi_\omega^{p-1}\mbox{Im} \xi_{-1}=\tilde\lambda_{-1}\Big(\frac{\omega^2}{1-\tilde\lambda_{-1}}+1\Big)\xi_{-1},\\
           \eta_{-1}=\frac{i\omega}{1-\tilde\lambda_{-1}}\xi_{-1}.
                    \end{array}
                  \right.
\end{align*}
Now we use facts \eqref{0.03} and \eqref{0.04}, to obtain that
\begin{align}
           \tilde\lambda_{-1}\Big(\frac{\omega^2}{1-\tilde\lambda_{-1}}+1\Big)=\lambda_{-1}, \quad \mbox{and}\quad  \xi_{-1}=f_{-1}.\label{14.56}
\end{align}
Then we find that given $\lambda_{-1}<0$, there exactly exists one negative solution $\tilde\lambda_{-1}<0$, satisfying the first equation in \eqref{14.56}. This implies $S_{\omega}''(\overrightarrow{\Phi_\omega})$ has exactly one negative eigenvalue. That is, there uniquely exists $(\tilde\lambda_{-1}, \vec \xi_{-1})$ satisfying \eqref{0.04}.
This proves \eqref{0.02}.

Now we are ready to prove the lemma.
Since $\phi_\omega$ is exponentially localized, $S_\omega''(\overrightarrow{\Phi_\omega})$ can be considered
as compact perturbation of
\begin{align*}
2\begin{pmatrix}
  -\partial_{xx}+1& i\omega\\
-i\omega&1
\end{pmatrix}.
\end{align*}
Therefore its essential spectrum is $[2(1-\omega^2),\infty)$ and
by Weyl's Theorem its spectrum in $(-\infty,2(1-\omega^2))$ consists of isolated
eigenvalues. Without loss of generality, we may assume that $\vec \xi_{-1}$ is the $L^2\times L^2$-normalized eigenvector associated to the negative eigenvalue $\lambda_{-1}$, that is
\begin{align}
S_{\omega}''(\overrightarrow{\Phi_\omega})\vec \xi_{-1}=\tilde\lambda_{-1} \vec \xi_{-1},\quad \mbox{and }\quad \|\vec\xi_{-1}\|_{L^2\times L^2}=1.\label{0.04'}
\end{align}
According to these, we may write the decomposition of $\vec \xi$ along the spectrum of $S_{\omega}''(\overrightarrow{\Phi_\omega})$ as
$$
\vec \xi = a_{-1}\vec\xi_{-1}+a_{0,1}i\overrightarrow{\Phi_\omega}+a_{0,2}\partial_x\overrightarrow{\Phi_\omega}+\vec \eta,
$$
with $a_{-1},a_{0,1}, a_{0,2}\in \R$, and $\vec \eta$ verifying
$\langle \vec \eta,\vec\xi_{-1}\rangle= \langle \vec \eta,i\overrightarrow{\Phi_\omega}\rangle=\langle \vec \eta,\partial_x\overrightarrow{\Phi_\omega}\rangle=0$ and
\begin{align}
\left\langle S_\omega''\big(\overrightarrow{\Phi_\omega}\big)\vec \eta, \vec \eta\right\rangle
\gtrsim \big\|\vec \eta\big\|_{H^1\times L^2}^2.
\label{17.37}
\end{align}
Since $\langle \vec \xi,i\overrightarrow{\Phi_\omega}\rangle=\langle \vec \xi,\partial_x\overrightarrow{\Phi_\omega}\rangle=0$, we have $a_{0,1}=a_{0,2}=0$, and thus
\begin{align}
\vec \xi = a_{-1}\vec\xi_{-1}+\vec \eta.
\label{17.36}
\end{align}
Similarly, noting that $\langle \vec {\psi_\omega},i\overrightarrow{\Phi_\omega}\rangle=\langle \vec {\psi_\omega},\partial_x\overrightarrow{\Phi_\omega}\rangle=0$, we write
\begin{align}
\vec{\psi_\omega}=b_{-1}\vec \xi_{-1}+\vec g,\label{17.28}
\end{align}
with $b_{-1}\in \R$ and $\vec g$ verifying
$$
\langle \vec g,\vec\xi_{-1}\rangle= \langle \vec g,i\overrightarrow{\Phi_\omega}\rangle=\langle \vec g,\partial_x\overrightarrow{\Phi_\omega}\rangle=0,
\,\,\mbox{and}\,\,\left\langle S_\omega''\big(\overrightarrow{\Phi_\omega}\big)\vec g, \vec g\right\rangle
\gtrsim \big\|\vec g\big\|_{H^1\times L^2}^2.
$$
From \eqref{17.36}, we find
\begin{align}
\left\langle S_\omega''\big(\overrightarrow{\Phi_\omega}\big)\vec \xi, \vec \xi\right\rangle
= \tilde\lambda_{-1}a_{-1}^2+\left\langle S_\omega''\big(\overrightarrow{\Phi_\omega}\big)\vec \eta, \vec \eta\right\rangle.\label{19.16}
\end{align}
Hence by \eqref{17.37}, we only need to estimate $\tilde\lambda_{-1}a_{-1}^2$.
To this end, we shall use the third orthogonality condition.

For simplicity, we denote
$$
\delta_0=-\left\langle S_\omega''\big(\overrightarrow{\Phi_\omega}\big)\vec \psi_\omega, \vec \psi_\omega\right\rangle,
$$
then from \eqref{negative}, we have $\delta_0>0$. Moreover, using \eqref{17.28} we obtain the relationship
\begin{align}
\tilde\lambda_{-1}b_{-1}^2=-\delta_0-\langle S_\omega''\big(\overrightarrow{\Phi_\omega}\big)\vec {g}, \vec {g}\rangle.\label{18.46}
\end{align}
Furthermore, the formulas \eqref{23.34} and \eqref{17.28} imply
$$
\overrightarrow{\Psi_\omega}=\tilde\lambda_{-1} b_{-1}\vec \xi_{-1}+S_\omega''\big(\overrightarrow{\Phi_\omega}\big)\vec g.
$$
Hence,  with combination of \eqref{17.36} and the orthogonality condition $\langle\vec \xi ,\overrightarrow{\Psi_\omega}\rangle=0$, we have
\begin{align}
\tilde\lambda_{-1}a_{-1}b_{-1}+\langle S_\omega''\big(\overrightarrow{\Phi_\omega}\big)\vec {g}, \vec \eta\rangle=0.\label{18.47}
\end{align}
Together with \eqref{18.46} and \eqref{18.47}, and using the Cauchy-Schwartz inequality, we obtain that
\begin{align}
-\tilde\lambda_{-1}a_{-1}^2&=\frac{\tilde\lambda_{-1}^2a_{-1}^2b_{-1}^2}{-\tilde\lambda_{-1}b_{-1}^2}
=\frac{\langle S_\omega''\big(\overrightarrow{\Phi_\omega}\big)\vec {g}, \vec \eta\rangle^2}{\delta_0+\langle S_\omega''\big(\overrightarrow{\Phi_\omega}\big)\vec {g}, \vec {g}\rangle}\notag\\
&\le  \frac{\langle S_\omega''\big(\overrightarrow{\Phi_\omega}\big)\vec \eta, \vec \eta\rangle \langle S_\omega''\big(\overrightarrow{\Phi_\omega}\big)\vec {g}, \vec g\rangle}{\delta_0+\langle S_\omega''\big(\overrightarrow{\Phi_\omega}\big)\vec {g}, \vec {g}\rangle}.\label{19.24}
\end{align}
Hence this combining with \eqref{19.16} and \eqref{17.37}, gives  
\begin{align}
\left\langle S_\omega''\big(\overrightarrow{\Phi_\omega}\big)\vec \xi, \vec \xi\right\rangle
&\ge - \frac{\langle S_\omega''\big(\overrightarrow{\Phi_\omega}\big)\vec \eta, \vec \eta\rangle \langle S_\omega''\big(\overrightarrow{\Phi_\omega}\big)\vec {g}, \vec g\rangle}{\delta_0+\langle S_\omega''\big(\overrightarrow{\Phi_\omega}\big)\vec {g}, \vec {g}\rangle}+\left\langle S_\omega''\big(\overrightarrow{\Phi_\omega}\big)\vec \eta, \vec \eta\right\rangle\notag\\
&= 
\delta_0\frac{\langle S_\omega''\big(\overrightarrow{\Phi_\omega}\big)\vec \eta, \vec \eta\rangle}{\delta_0+\langle S_\omega''\big(\overrightarrow{\Phi_\omega}\big)\vec {g}, \vec {g}\rangle}
\gtrsim \big\|\vec \eta\big\|_{H^1\times L^2}^2.\label{19.32}
\end{align}
Using \eqref{19.24} again, and by H\"older's inequality,  we have 
$$
a_{-1}^2\lesssim \big\|\vec \eta\big\|_{H^1\times L^2}^2.
$$
Hence, from \eqref{17.36}, 
$$
\big\|\vec \xi\big\|_{L^2\times L^2}^2\lesssim a_{-1}^2+ \big\|\vec \eta\big\|_{H^1\times L^2}^2
\lesssim \big\|\vec \eta\big\|_{H^1\times L^2}^2.
$$
This together with \eqref{19.32}, yields 
\begin{align}
\left\langle S_\omega''\big(\overrightarrow{\Phi_\omega}\big)\vec \xi, \vec \xi\right\rangle
\gtrsim \big\|\vec \xi\big\|_{L^2\times L^2}^2. \label{21.07}
\end{align}
Lastly, from the definition of $S_\omega''\big(\overrightarrow{\Phi_\omega}\big)$ in \eqref{23.06}, we have 
\begin{align*}
\big\|\vec \xi\big\|_{H^1\times L^2}^2
&\lesssim\langle S_\omega''\big(\overrightarrow{\Phi_\omega}\big)\vec \xi, \vec \xi\rangle+\|\vec \xi \|_{L^2\times L^2}^2.
\end{align*}
Therefore, followed from \eqref{21.07}, we obtain that 
\begin{align}
\left\langle S_\omega''\big(\overrightarrow{\Phi_\omega}\big)\vec \xi, \vec \xi\right\rangle
\gtrsim \big\|\vec \xi\big\|_{H^1\times L^2}^2. 
\end{align}
This finishes the proof of the lemma.
\end{proof}

\subsection{Modulation}
The following modulation lemma says that if the standing wave solution is stable, then  after suitably choosing the parameters, the orthogonality conditions in Lemma \ref{lem:Coercivity} can be verified.
\begin{lem}\label{lem:modulation}
Let $\omega=\pm\omega_c$. There exists $\varepsilon_0>0$, such that for any $\varepsilon\in(0,\varepsilon_0)$, if $\vec u(t)\in U_\varepsilon(\vec\Phi_\omega)$ for any $t\in\R$, then the following properties is verified. There exist $C^1$-functions
$$
(\theta,y): \R^2\rightarrow \R,\quad
\lambda: \R\rightarrow \R^+,
$$
such that if we define $\vec\xi$ by
\begin{align}
\vec\xi(t)=e^{-i\theta(t)}\vec u\big(t,\cdot-y(t)\big)-\overrightarrow{\Phi_{\lambda(t)\omega}},\label{modulation-u}
\end{align}
then $\vec \xi$ satisfies the following orthogonality conditions for any $t\in\R$,
\begin{align}\label{orth-condition}
\left\langle\vec \xi,i\overrightarrow{\Phi_{\lambda(t)\omega}}\right\rangle=\left\langle\vec\xi,\partial_x\overrightarrow{\Phi_{\lambda(t)\omega}}\right\rangle
=\left\langle\vec\xi,\overrightarrow{\Psi_{\lambda(t)\omega}}\right\rangle=0.
\end{align}
Moreover, the following estimates verify
\begin{align*}
 \|\vec\xi\|_{H^1\times L^2}+|\lambda-1|\lesssim \varepsilon,
\end{align*}
and for any $t\in \R$,
\begin{align*}
 |\dot \theta-\lambda(t)\omega|+|\dot y|+|\dot\lambda|=O\big( \big\|\vec \xi\big\|_{H^1\times L^2}\big).
\end{align*}
\end{lem}
\begin{proof}
Since the argument is standard, see c.f. Proposition 1 in \cite{MaMe-GFA-2001} and Proposition 9 in \cite{BeGhLecoz-CPDE-14}, we give the proof much briefly.
The existence of the parameters follows from classical arguments
involving the implicit function theorem. More precisely, fixing $t\in \R$ and writting $\vec u=\vec u(t)$ for short, 
we denote $F_j, j=1,2,3: U_1(\vec\Phi_\omega)\times \R\times \R\times \R^+$ by 
$$
F_1(\vec u,\theta,y,\lambda)=\left\langle\vec \xi,i\overrightarrow{\Phi_{\lambda\omega}}\right\rangle;\,\,
F_2(\vec u,\theta,y,\lambda)=\left\langle\vec\xi,\partial_x\overrightarrow{\Phi_{\lambda\omega}}\right\rangle;\,\,
F_3(\vec u,\theta,y,\lambda)=\left\langle\vec\xi,\overrightarrow{\Psi_{\lambda\omega}}\right\rangle.
$$ 
Then 
$$
F_j\big(\overrightarrow{\Phi_{\lambda\omega}}, 0, 0,1\big)=0,\quad \mbox{for } j=1,2,3.
$$
Moreover, a direct computation gives that 
\begin{align*}
&\left|\begin{array}{ccc}
\partial_\theta F_1 & \partial_y F_1 & \partial_\lambda F_1\\
\partial_\theta F_2 & \partial_y F_2 & \partial_\lambda F_2\\
\partial_\theta F_3 & \partial_y F_3 & \partial_\lambda F_3\\
\end{array}\right|_{(\vec u,\theta,y,\lambda)=\big(\overrightarrow{\Phi_{\lambda\omega}}, 0, 0,1\big)}\\
&\qquad =
\left|\begin{array}{ccc}
-\big\|\overrightarrow{\Phi_{\omega}}\big\|_{L^2\times L^2}& 0 & 0\\
0 & -\big\|\partial_x\overrightarrow{\Phi_{\omega}}\big\|_{L^2\times L^2} & 0\\
0 & 0 & 2\|\phi_\omega\|_{L^2}^2\\
\end{array}\right|\neq0.
\end{align*}
Therefore, the implicit function theorem implies that there exists $\varepsilon_0>0$, such that
for any $\varepsilon\in(0,\varepsilon_0)$, for any $\vec u\in U_\varepsilon(\vec\Phi_\omega)$, there exist continuity functions
$$
(\theta,y): U_\varepsilon(\vec\Phi_\omega)\rightarrow \R^2,\quad
\lambda: U_\varepsilon(\vec\Phi_\omega)\rightarrow \R^+,
$$
such that $F_j(\vec u,\theta,y,\lambda)=0$ for $j=1,2,3$. 

The parameters $(\theta, y,\lambda)\in C^1$ in time can be followed from the regularization arguments, see c.f. Lemma 4 in \cite{MaMe-GFA-2001}.
Now we consider the dynamic of the parameters. From \eqref{modulation-u}, we have 
$$
\vec u(t)=e^{i\theta(t)}\big(\vec \xi+\overrightarrow{\Phi_{\lambda(t)\omega}}\big)\big(t,\cdot+y(t)\big).
$$
Then using this equality, the equations
$$
u_t=v,\quad v_t=\Delta u-u+|u|^{p-1}u,
$$ 
and \eqref{Elliptic}, we obtain that 
\begin{align}
\partial_t\vec\xi+i(\dot \theta-\lambda\omega)\big(\vec \xi+\overrightarrow{\Phi_{\lambda(t)\omega}}\big)+\dot y\partial_x\big(\vec \xi+\overrightarrow{\Phi_{\lambda(t)\omega}}\big)+\dot \lambda\omega\partial_\lambda \overrightarrow{\Phi_{\lambda(t)\omega}}=\mathcal N(\vec \xi).
\label{0.49}
\end{align}
Here we have used the notations $\dot f=\partial_tf$ for the time dependent function $f$, and $\mathcal N(\vec \xi)$ verifying 
$$
\big\langle\mathcal N(\vec\xi),\vec f\big\rangle= O\big( \big\|\vec \xi\big\|_{H^1\times L^2}\big)\big\|\vec f\big\|_{H^1\times L^2},\quad \mbox{ for any }f\in H^1\times L^2.
$$
Now multiplying \eqref{0.49} by $i\overrightarrow{\Phi_{\lambda(t)\omega}}$, $\partial_x\overrightarrow{\Phi_{\lambda(t)\omega}}$ and $\overrightarrow{\Psi_{\lambda(t)\omega}}$, respectively,  integrating by parts and then using the orthogonal conditions \eqref{orth-condition}, we obtain that 
\begin{align*}
(\dot \theta-\lambda\omega)\Big(\|\overrightarrow{\Phi_{\lambda\omega}}\|_{L^2\times L^2}^2+\langle\vec \xi,\overrightarrow{\Phi_{\lambda\omega}}\rangle\Big)+\dot y\langle\partial_x\vec \xi,i\overrightarrow{\Phi_{\lambda\omega}}\rangle-\dot \lambda\omega\langle\vec \xi,i\partial_\lambda\overrightarrow{\Phi_{\lambda\omega}}\rangle=O\big( \big\|\vec \xi\big\|_{H^1\times L^2}\big);
\end{align*}
\begin{align*}
(\dot \theta-\lambda\omega)&\langle i\vec \xi,\partial_x\overrightarrow{\Phi_{\lambda\omega}}\rangle+\dot y\Big(\|\partial_x\overrightarrow{\Phi_{\lambda\omega}}\|_{L^2\times L^2}^2+\langle\partial_x\vec \xi,\partial_x\overrightarrow{\Phi_{\lambda\omega}}\rangle\Big)\\
&-\dot \lambda\omega\langle\vec \xi,\partial_x\partial_\lambda\overrightarrow{\Phi_{\lambda\omega}}\rangle=O\big( \big\|\vec \xi\big\|_{H^1\times L^2}\big);
\end{align*}
and 
\begin{align*}
(\dot \theta-\lambda\omega)&\langle i\vec \xi,\overrightarrow{\Psi_{\lambda\omega}}\rangle+\dot y\langle\partial_x\vec \xi,\overrightarrow{\Psi_{\lambda\omega}}\rangle\\
&-\dot \lambda\omega\Big(2\|\phi_{\lambda\omega}\|_{L^2}^2+\langle\vec \xi,\partial_\lambda\overrightarrow{\Psi_{\lambda\omega}}\rangle\Big)=O\big( \big\|\vec \xi\big\|_{H^1\times L^2}\big);
\end{align*}
With combination of these three estimates, we obtain that 
\begin{align*}
 |\dot \theta-\lambda\omega|+|\dot y|+|\dot\lambda|=O\big( \big\|\vec \xi\big\|_{H^1\times L^2}\big).
\end{align*}
This finishes the proof of the lemma.
\end{proof}

\vskip 2cm

\section{Proof of the main theorem}
\vskip 0.2cm

\subsection{Localized virial identities}
To prove main Theorem \ref{thm:main1}, one of the key ingredient is the localized virial identities.
\begin{lem}\label{lem:virial}
Let $\varphi\in C^1(\R)$, then
\begin{align*}
\frac{d}{dt} \mbox{Re} \int  u\bar u_t\,dx
&=\int\big[|u_t|^2-|u_x|^2-|u|^2+|u|^{p+1}\big]\,dx;\\
\mbox{Re} \int \varphi \frac{d}{dt} \Big(u_x\bar u_t\Big)\,dx
&=-\frac12\int\varphi'\big[|u_t|^2+|u_x|^2-|u|^2+\frac{2}{p+1}|u|^{p+1}\big]\,dx.
\end{align*}
\end{lem}
\begin{proof}
It follows from a direct calculation. See \cite{OhtaTod-SIAM-07} for the details.
\end{proof}

Now we define the smooth cutoff function $\varphi_R\in C^\infty(\R)$ as
\begin{align*}
\varphi_R(x)=x, \quad \mbox{when} \quad |x|\le  R;\qquad \varphi_R(x)=0,\quad  \mbox{when} \quad |x|\ge 2R,
\end{align*}
and $0\le \varphi_R'\le 1$ for any $x\in \R$. Moreover, we denote
$$
I(t)=\frac4{p-1}\mbox{Re}\int u\bar u_t\,dx+2\mbox{Re} \int \varphi_R\big(x-y(t)\big) u_x\bar u_t\,dx.
$$
Then from Lemma \ref{lem:virial} we have the following lemma.
\begin{lem}\label{lem:general-Virial} Let $R>0$, if $|\dot y|\lesssim 1$, then
\begin{align*}
I'(t)=&-\frac{p+3}{p-1}\cdot 2E(u_0,u_1)-\frac{16\omega}{p-1}Q(u_0,u_1)-2\dot yP(u_0,u_1)+\frac8{p-1}\|u_t-i\omega u\|_{L^2}^2
\\
&\qquad+O\Big(\int_{|x-y(t)|\ge R} |u_t|^2+|u_x|^2+|u|^2+|u|^{p+1}\,dx\Big).
\end{align*}
\end{lem}
\begin{proof}
First, we have
\begin{align*}
\frac{d}{dt}\mbox{Re} \int \varphi_R\big(x-y(t)\big)& u_x\bar u_t\,dx
=-\dot y\mbox{Re} \int \varphi_R'\big(x-y(t)\big) u_x\bar u_t\,dx\\
& +\mbox{Re} \int \varphi_R\big(x-y(t)\big)\frac{d}{dt}\big( u_x\bar u_t\big)\,dx.
\end{align*}
Then from Lemma \ref{lem:virial} and the momentum  conservation law, we obtain
\begin{align*}
\frac{d}{dt}\mbox{Re} &\int \varphi_R\big(x-y(t)\big) u_x\bar u_t\,dx
=-\dot y\mbox{Re} \int \varphi_R'\big(x-y(t)\big) u_x\bar u_t\,dx\\
&\qquad  -\frac12\int\varphi_R'\big(x-y(t)\big)\big[|u_t|^2+|u_x|^2-|u|^2+\frac{2}{p+1}|u|^{p+1}\big]\,dx\\
&=-\dot y P(u_0,u_1)-\dot y\mbox{Re} \int \big[\varphi_R'\big(x-y(t)\big)-1\big] u_x\bar u_t\,dx\\
&\quad  -\frac12\int \big(|u_t|^2+|u_x|^2-|u|^2+\frac{2}{p+1}|u|^{p+1}\big)\,dx\\
&\qquad  -\frac12\int\big[\varphi_R'\big(x-y(t)\big)-1\big]\big(|u_t|^2+|u_x|^2-|u|^2+\frac{2}{p+1}|u|^{p+1}\big)\,dx.
\end{align*}
Since supp$\big[\varphi_R'\big(x-y(t)\big)-1\big]\subset \{x:|x-y(t)|\ge R\}$, $0\le \varphi_R'\le 1$ and $|\dot y|\lesssim 1$, we get
\begin{align*}
\frac{d}{dt}\mbox{Re} &\int \varphi_R\big(x-y(t)\big) u_x\bar u_t\,dx\\
&=-\dot y P(u_0,u_1)  -\frac12\int \big(|u_t|^2+|u_x|^2-|u|^2+\frac{2}{p+1}|u|^{p+1}\big)\,dx\\
&\qquad  +O\Big(\int_{|x-y(t)|\ge R} |u_t|^2+|u_x|^2+|u|^2+|u|^{p+1}\,dx\Big).
\end{align*}
Moreover, from Lemma \ref{lem:virial},
$$
\frac{d}{dt} \mbox{Re} \int  u\bar u_t\,dx
=\int\big[|u_t|^2-|u_x|^2-|u|^2+|u|^{p+1}\big]\,dx.
$$
Combining the two estimates above, we obtain that
\begin{align}
I'(t)=&-2\dot yP(u_0,u_1)+\Big(\frac4{p-1}-1\Big)\|u_t\|_{L^2}^2-\frac{p+3}{p-1}\|u_x\|_{L^2}^2\notag\\
&\qquad+\frac{p-5}{p-1}\|u\|_{L^2}^2+2\frac{p+3}{p^2-1}\|u\|_{L^{p+1}}^{p+1}\notag\\
&\qquad\qquad+O\Big(\int_{|x-y(t)|\ge R} |u_t|^2+|u_x|^2+|u|^2+|u|^{p+1}\,dx\Big).\label{2.35}
\end{align}
Note that when $|\omega|=\omega_c$,
\begin{align*}
\Big(\frac4{p-1}-1\Big)&\|u_t\|_{L^2}^2-\frac{p+3}{p-1}\|u_x\|_{L^2}^2+\frac{p-5}{p-1}\|u\|_{L^2}^2+2\frac{p+3}{p^2-1}\|u\|_{L^{p+1}}^{p+1}\\
& =\frac8{p-1}\|u_t-i\omega u\|_{L^2}^2-\frac{p+3}{p-1}\cdot 2E(u_0,u_1)-\frac{16\omega}{p-1}Q(u_0,u_1).
\end{align*}
Inserting this equality into \eqref{2.35}, we prove the lemma.
\end{proof}

\subsection{The choice of the initial data}
In this subsection, we choose the initial data such that it is close to the standing waves solution but leads the instability. We set
\begin{align}
\vec u_0=(1+a)\overrightarrow{\Phi_\omega},\label{ID-choice}
\end{align}
here $a\in (0,a_0)$ is an arbitrary small constant, and $a_0$ will be decided later.
Then we have
\begin{lem}\label{lem:Charge} Let $\vec u_0$ be defined in \eqref{ID-choice}, then
$$
P(\vec u_0)=0,
$$
and
$$
Q(\vec u_0)-Q\left(\overrightarrow{\Phi_\omega}\right)=-2a\omega \|\phi_\omega\|_{L^2}^2+O(a^2).
$$
\end{lem}
\begin{proof}
It follows from the definition that $P(\vec u_0)=0$. Now consider $Q(\vec u_0)$.
We write
\begin{align*}
Q(\vec u_0)-Q\left(\overrightarrow{\Phi_\omega}\right)
=&\left\langle Q'\left(\overrightarrow{\Phi_\omega}\right),\vec u_0- \overrightarrow{\Phi_\omega}\right\rangle +O\big(\|\vec u_0- \overrightarrow{\Phi_\omega}\|_{H^1\times L^2}^2\big)\\
=&-\omega\left\langle \phi_\omega, u_0-  \phi_\omega\right\rangle-\left\langle i\phi_\omega, u_1-  i\omega\phi_\omega\right\rangle +O\big(a^2\big)\\
=&-2a\omega \|\phi_\omega\|_{L^2}^2+O(a^2).
\end{align*}
This finishes the proof of the lemma.
\end{proof}

Using the lemma above, we can scale the main part in $I'(t)$.
\begin{lem}\label{lem:mainpart} Let $\vec u_0$ be defined in \eqref{ID-choice}, then
$$
-\frac{p+3}{p-1}\cdot 2 E(\vec u_0)-\frac{16\omega}{p-1}Q(\vec u_0)
=\frac{5-p}{p-1}\cdot 4a\omega^2\|\phi_\omega\|_{L^2}^2+O(a^2).
$$
\end{lem}
\begin{proof}
Making use of Lemma \ref{lem:QE-Phi} (2), we have
\begin{align*}
-\frac{p+3}{p-1}&\cdot 2 E(\vec u_0)-\frac{16\omega}{p-1}Q(\vec u_0)\\
=&-\frac{p+3}{p-1}\cdot 2 \Big[E(\vec u_0)-E\left(\overrightarrow{\Phi_\omega}\right)\Big]-\frac{16\omega}{p-1}\Big[Q(\vec u_0)-Q\left(\overrightarrow{\Phi_\omega}\right)\Big]\\
&\qquad -\frac{p+3}{p-1}\cdot 2 E\left(\overrightarrow{\Phi_\omega}\right)-\frac{16\omega}{p-1}Q\left(\overrightarrow{\Phi_\omega}\right)\\
=&-\frac{p+3}{p-1}\cdot 2 \Big[E(\vec u_0)-E\left(\overrightarrow{\Phi_\omega}\right)\Big]-\frac{16\omega}{p-1}\Big[Q(\vec u_0)-Q\left(\overrightarrow{\Phi_\omega}\right)\Big].
\end{align*}
Since
$$
E(\vec u_0)-E\left(\overrightarrow{\Phi_\omega}\right)
=\Big[S(\vec u_0)-S\left(\overrightarrow{\Phi_\omega}\right)\Big]-\omega\Big[Q(\vec u_0)-Q\left(\overrightarrow{\Phi_\omega}\right)\Big],
$$
we further write
\begin{align*}
-\frac{p+3}{p-1}&\cdot 2 E(\vec u_0)-\frac{16\omega}{p-1}Q(\vec u_0)\\
=&-\frac{p+3}{p-1}\cdot 2 \Big[S(\vec u_0)-S\left(\overrightarrow{\Phi_\omega}\right)\Big]-\frac{5-p}{p-1}\cdot 2\omega\Big[Q(\vec u_0)-Q\left(\overrightarrow{\Phi_\omega}\right)\Big].
\end{align*}
By Taylor's type extension, we have
$$
S(\vec u_0)-S\left(\overrightarrow{\Phi_\omega}\right)=O\big(\|\vec u_0- \overrightarrow{\Phi_\omega}\|_{H^1\times L^2}^2\big)=O(a^2).
$$
Now using Lemma \ref{lem:Charge}, we  prove the lemma.
\end{proof}
Similar computation also gives
\begin{lem}\label{lem:S-u0} Let $\lambda \in \R^+$ with $\lambda\lesssim 1$, $\vec u_0$ be defined in \eqref{ID-choice}, then
$$
S_{\lambda \omega}(\vec u_0)-S_{\lambda \omega}\left(\overrightarrow{\Phi_\omega}\right)
= -2(\lambda-1)a\omega^2\|\phi_\omega\|_{L^2}^2+O(a^2).
$$
\end{lem}
\begin{proof}
By the definition of $S_{\omega}$, we have
\begin{align*}
S_{\lambda \omega}&(\vec u_0)-S_{\lambda \omega}\left(\overrightarrow{\Phi_\omega}\right)\\
=&S_{\omega}(\vec u_0)-S_{ \omega}\left(\overrightarrow{\Phi_\omega}\right)+(\lambda-1)\omega\Big[Q_{\omega}(\vec u_0)-Q_{ \omega}\left(\overrightarrow{\Phi_\omega}\right)\Big].
\end{align*}
Since
$$
S(\vec u_0)-S\left(\overrightarrow{\Phi_\omega}\right)=O(a^2),
$$
then by Lemma \ref{lem:Charge}, we prove the lemma.
\end{proof}

Now we control the rest terms in the virial identity in Lemma \ref{lem:general-Virial}.
We argue for contradiction and suppose that the standing wave solution $u_\omega$ is stable. That is, for any $\varepsilon>0$, there exists a constant $a_0>0$, such that for any $a\in (0,a_0)$, if $\vec u_0\in U_a(\overrightarrow{\Phi_\omega})$, then $\vec u(t)\in U_\varepsilon(\overrightarrow{\Phi_\omega})$ for any $t\in\R$. We may assume that $\vec u\in U_\varepsilon(\vec\Phi_\omega)$ for $\varepsilon\le \varepsilon_0$, where $\varepsilon_0$ is determined in Lemma \ref{lem:modulation}. Hence by Lemma \ref{lem:modulation}, we can write
\begin{align}\label{9.57}
u=e^{i\theta}(\phi_{\lambda\omega}+\xi)(\cdot -y);\qquad
u_t=e^{i\theta}(i\lambda\omega\phi_{\lambda\omega}+\eta)(\cdot -y)
\end{align}
with  $\vec \xi=(\xi,\eta)$ satisfying the orthogonal conditions \eqref{orth-condition}.

\subsection{Lower control of $\|u_t-i\omega u\|_{L^2}$}
In this subsection, we prove the following lemma.
\begin{lem}\label{lem:control-lambda}
Suppose that $\vec \xi=(\xi,\eta)$ defined in \eqref{9.57} satisfying the orthogonal conditions \eqref{orth-condition}, then
\begin{align*}
\|u_t-i\omega u\|_{L^2}^2=&(\lambda-1)^2\omega^2\|\phi_\omega\|_{L^2}^2+\|\eta-i\omega\xi\|_{L^2}^2\\
&\qquad+O\Big(|\lambda-1|^3+a|\lambda-1|+\|\vec \xi\|_{H^1\times L^2}^3\Big).
\end{align*}
\end{lem}
\begin{proof}
By \eqref{9.57}, we expand it as
\begin{align*}
\|u_t-i\omega u\|_{L^2}^2=&\|i\lambda\omega\phi_{\lambda\omega}+\eta-i\omega(\phi_{\lambda \omega}+\xi)\|_{L^2}^2\\
=&\|i(\lambda-1)\omega\phi_{\lambda\omega}+\eta-i\omega\xi\|_{L^2}^2\\
=&(\lambda-1)^2\omega^2\|\phi_{\lambda\omega}\|_{L^2}^2+2(\lambda-1)\omega\big\langle \eta-i\omega \xi, i\phi_{\lambda\omega}\big\rangle+\|\eta-i\omega\xi\|_{L^2}^2.
\end{align*}
Noting that
$$
\|\phi_{\lambda\omega}\|_{L^2}^2=\|\phi_{\omega}\|_{L^2}^2+O(|\lambda-1|),
$$
then combining with the third orthogonal condition in \eqref{orth-condition}, we further get
\begin{align}
\|u_t-i\omega u\|_{L^2}^2
=&(\lambda-1)^2\omega^2\|\phi_{\omega}\|_{L^2}^2+2(\lambda-1)\omega\big\langle \eta, i\phi_{\lambda\omega}\big\rangle\notag\\
&\qquad+\|\eta-i\omega\xi\|_{L^2}^2+O(|\lambda-1|^3).\label{11.09}
\end{align}
Now we consider the term $\langle \eta, i\phi_{\lambda\omega}\rangle$. First, we use the charge conservation law to obtain
\begin{align*}
Q\big(\vec u_0\big)-Q\left(\overrightarrow{\Phi_\omega}\right)+&Q\left(\overrightarrow{\Phi_\omega}\right)-Q\left(\overrightarrow{\Phi_{\lambda\omega}}\right)\\
&=Q\big(\vec u\big)-Q\left(\overrightarrow{\Phi_{\lambda\omega}}\right)\\
&=-\big\langle \xi, \lambda\omega\phi_{\lambda\omega}\big\rangle-\big\langle \eta, i\phi_{\lambda\omega}\big\rangle+O\Big(\|\vec \xi\|_{H^1\times L^2}^2\Big).
\end{align*}
Then by  the third orthogonal conditions in \eqref{orth-condition}, we have
\begin{align*}
\big\langle \eta, i\phi_{\lambda\omega}\big\rangle
&=Q\left(\overrightarrow{\Phi_{\lambda\omega}}\right)-Q\left(\overrightarrow{\Phi_\omega}\right)
-\Big[Q\big(\vec u_0\big)-Q\left(\overrightarrow{\Phi_\omega}\right)\Big]+O\Big(\|\vec \xi\|_{H^1\times L^2}^2\Big).
\end{align*}
From Lemma \ref{lem:QE-Phi}, we have
$$
Q\left(\overrightarrow{\Phi_{\lambda\omega}}\right)-Q\left(\overrightarrow{\Phi_\omega}\right)=O\big(|\lambda-1|^2\big),
$$
and from Lemma \ref{lem:Charge}, we have
$$
Q\big(\vec u_0\big)-Q\left(\overrightarrow{\Phi_\omega}\right)=O(a).
$$
Therefore, we obtain that
\begin{align}
\big\langle \eta, i\phi_{\lambda\omega}\big\rangle
=O\Big(a+|\lambda-1|^2+\|\vec \xi\|_{H^1\times L^2}^2\Big).\label{11.10}
\end{align}
Now together \eqref{11.09} with \eqref{11.10}, we obtain the desirable result.
\end{proof}
\subsection{Upper control of $\|\vec \xi\|_{H^1\times L^2}$}
In this subsection, we give the following estimate on $\|\vec \xi\|_{H^1\times L^2}$.
\begin{lem}\label{lem:Remainer} Let $\vec \xi=(\xi,\eta)$ be defined in \eqref{9.57}, then
\begin{align*}
\|\vec \xi\|_{H^1\times L^2}^2&
=O(a|\lambda-1|+a^2)+o\big((\lambda-1)^2\big).
\end{align*}
\end{lem}
\begin{proof}
From the charge and energy conservation laws,
\begin{align*}
S_{\lambda \omega}\big(\vec u_0\big)&=S_{\lambda \omega}\big(\vec u\big)\\
&=S_{\lambda \omega}\big(\vec u\big)-S_{\lambda \omega}\left(\overrightarrow{\Phi_{\lambda\omega}}\right)+S_{\lambda \omega}\left(\overrightarrow{\Phi_{\lambda\omega}}\right)\\
&=\frac12\left\langle S_{\lambda \omega}''\left(\overrightarrow{\Phi_{\lambda\omega}}\right)\vec \xi, \vec \xi\right\rangle+S_{\lambda \omega}\left(\overrightarrow{\Phi_{\lambda\omega}}\right)+o\big(\|\vec \xi\|_{H^1\times L^2}^2\big).
\end{align*}
Hence by Lemma \ref{lem:Coercivity},
\begin{align*}
\|\vec \xi\|_{H^1\times L^2}^2&\lesssim \frac12\left\langle S_{\lambda \omega}''\left(\overrightarrow{\Phi_{\lambda\omega}}\right)\vec \xi, \vec \xi\right\rangle\\
&=\Big[S_{\lambda \omega}\big(\vec u_0\big)-S_{\lambda \omega}\left(\overrightarrow{\Phi_{\omega}}\right)\Big]-\Big[S_{\lambda \omega}\left(\overrightarrow{\Phi_{\lambda\omega}}\right)-S_{\lambda \omega}\left(\overrightarrow{\Phi_{\omega}}\right)\Big]+o\big(\|\vec \xi\|_{H^1\times L^2}^2\big).
\end{align*}
By Lemma \ref{lem:S-u0},
$$
S_{\lambda \omega}(\vec u_0)-S_{\lambda \omega}\left(\overrightarrow{\Phi_\omega}\right)
= -2(\lambda-1)a\omega^2\|\phi_\omega\|_{L^2}^2+O(a^2),
$$
and by Corollary  \ref{cor:S-lambda},
\begin{align*}
S_{\lambda\omega}\big(\overrightarrow{\Phi_{\lambda\omega}}\big)-S_{\lambda\omega}\big(\overrightarrow{\Phi_{\omega}}\big)=o\big((\lambda-1)^2\big).
\end{align*}
Therefore,
\begin{align*}
\|\vec \xi\|_{H^1\times L^2}^2&
=O(a|\lambda-1|+a^2)+o\big((\lambda-1)^2\big)+o\big(\|\vec \xi\|_{H^1\times L^2}^2\big).
\end{align*}
Absorbing the last term by the left-hand side one, we prove the lemma.
\end{proof}

\subsection{Proof of Theorem \ref{thm:main1}}
As discussion above, we assume that $\vec u\in U_\varepsilon(\vec\Phi_\omega)$, and thus $|\lambda-1|\lesssim \varepsilon$. First, we note that from the definition of $I(t)$, we have the time uniform boundedness of $I(t)$,
\begin{align}
\sup\limits_{t\in\R} I(t)
\lesssim R\Big(\|\overrightarrow{\Phi_{\omega}}\|_{H^1\times L^2}^2+1\Big).\label{bound-It}
\end{align}
Now we consider the estimate on $I'(t)$. First, by \eqref{9.57}, the exponential decaying of $\phi_\omega$ and $\frac12\le \lambda\le \frac32$,
\begin{align*}
\int_{|x-y(t)|\ge R}&\Big[ |u_t|^2+|u_x|^2+|u|^2+|u|^{p+1}\Big]\,dx\\
&\lesssim \int_{|x|\ge R}\Big[ |\phi_{\lambda\omega}|^2+|\partial_x\phi_{\lambda\omega}|^2+|\xi|^2+|\partial_x\xi|^2+|\xi|^{p+1}+|\eta|^2\Big]\,dx\\
&=O\big(\big\|\vec \xi\big\|_{H^1\times L^2}^2+\frac1R\big).
\end{align*}
Hence by Lemma \ref{lem:general-Virial},
\begin{align*}
I'(t)=&-\frac{p+3}{p-1}\cdot 2E(u_0,u_1)-\frac{16\omega}{p-1}Q(u_0,u_1)
\\
&\qquad-2\dot yP(u_0,u_1)+\frac8{p-1}\|u_t-i\omega u\|_{L^2}^2+O\Big(\big\|\vec \xi\big\|_{H^1\times L^2}^2+\frac1R\Big).
\end{align*}
Now by Lemma \ref{lem:mainpart}, Lemma \ref{lem:Charge}, and Lemma \ref{lem:control-lambda}, we have
\begin{align*}
I'(t)=&\frac{5-p}{p-1}\cdot 4a\omega^2\|\phi_\omega\|_{L^2}^2+(\lambda-1)^2\omega^2\|\phi_\omega\|_{L^2}^2+\|\eta-i\omega\xi\|_{L^2}^2
\\
&\qquad +O\Big(a^2+a|\lambda-1|+|\lambda-1|^3+\big\|\vec \xi\big\|_{H^1\times L^2}^2+\frac1R\Big).
\end{align*}
Using Lemma \ref{lem:Remainer}, we further get
\begin{align*}
I'(t)=&\frac{5-p}{p-1}\cdot 4a\omega^2\|\phi_\omega\|_{L^2}^2+(\lambda-1)^2\omega^2\|\phi_\omega\|_{L^2}^2+\|\eta-i\omega\xi\|_{L^2}^2
\\
&\qquad +O\big(a^2+a|\lambda-1|\big)+o\big(|\lambda-1|^2\big).
\end{align*}
Choosing $\varepsilon$ and $a_0$ small enough, we obtain that for any $a\in (0,a_0)$,
\begin{align*}
I'(t)\ge &\frac{5-p}{p-1}\cdot 2a\omega^2\|\phi_\omega\|_{L^2}^2.
\end{align*}
This implies that $I(t)\to +\infty$ when $t\to +\infty$, which is contradicted with \eqref{bound-It}.
Hence we prove the instability of the standing wave $u_\omega$ and thus give the proof of Theorem \ref{thm:main1}.

\vskip .4in
\section*{Acknowledgements}
The author was partially supported by the NSFC (No. 11771325, 11571118). 
The author would also like to express his deep gratitude to  Professor Masaya Maeda for his helpful private discussion, in particular, he introduced me the references \cite{CoPe-CPAM-03, Maeda-JFA-12, Ohta-JFA-11}, and informed me a flaw in \cite{Maeda-JFA-12} such that the correlative theorem only applies to the nonlinear Klein-Gordon equation in cases of $p\ge 2$.

\end{document}